\documentclass[11pt]{amsart}

\usepackage{epigamath}

\usepackage[english]{babel}

\numberwithin{equation}{section}

\newtheorem{Theorem}{Theorem}[section]
\newtheorem{Lemma}[Theorem]{Lemma}
\newtheorem{Proposition}[Theorem]{Proposition}
\newtheorem{Corollary}[Theorem]{Corollary}
\newtheorem{Conjecture}[Theorem]{Conjecture}

\theoremstyle{definition}
\newtheorem{Definition}[Theorem]{Definition}
\newtheorem{Construction}[Theorem]{Construction}

\theoremstyle{remark}

\newtheorem{Remark}[Theorem]{Remark}

\usepackage[shortlabels]{enumitem}

\usepackage{mathtools}
\usepackage{amsmath}
\usepackage{amssymb}  
\usepackage{latexsym} 
\usepackage{comment}
\usepackage{amssymb,amsmath,amsthm,amsfonts,mathrsfs}
\usepackage[alphabetic,initials, short-journals, lite]{amsrefs}
\usepackage{amscd}   
\usepackage[all, cmtip]{xy} 
\usepackage{xfrac}
\usepackage[T1]{fontenc}

\DeclareFontEncoding{OT2}{}{} 

\usepackage[all]{xy}

\newcommand{\defi}[1]{\emph{#1}} 

\def\act#1#2%
  {\mathop{}%
   \mathopen{\vphantom{#2}}^{#1}%
   \kern-3\scriptspace%
   #2}

\newcommand{\F}{{\mathbb F}}
\newcommand{\A}{{\mathbb A}}

\newcommand{\Cbar}{{\overline{C}}}
\newcommand{\Dbar}{{\overline{D}}}

\newcommand{\Fbar}{{\overline{\F}}}

\newcommand{\xbar}{{\overline{x}}}
\newcommand{\ybar}{{\overline{y}}}

\newcommand{\calO}{{\mathcal O}}

\newcommand{\To}{\longrightarrow}

\DeclareMathOperator{\Map}{Map}

\DeclareMathOperator{\Hom}{Hom}

\DeclareMathOperator{\Gal}{Gal}

\DeclareMathOperator{\Br}{Br}

\DeclareMathOperator{\HH}{H}

\DeclareMathOperator{\Spec}{Spec}

\DeclareMathOperator{\Mor}{Mor}

\newcommand{\sep}{\operatorname{s}}
\newcommand{\unr}{\operatorname{unr}}
\newcommand{\etale}{\operatorname{\textup{\'et}}}
\newcommand{\wb}{\textup{wb}}

\makeatletter
\renewcommand{\BibLabel}{%
    \Hy@raisedlink{\hyper@anchorstart{cite.\CurrentBib}\hyper@anchorend}%
    [\thebib]%
}
\makeatother

\BibSpec{arXiv}{%
  +{}{\PrintAuthors}{author}
  +{,}{ \textit}{title}
  +{,}{  }{eprint}
  +{}{ \parenthesize}{date}
  +{.}{}{transition}
}

\def\lra{\longrightarrow}
\newcommand{\longhookrightarrow}{\xhookrightarrow{\hphantom{aaa}}}

\EpigaVolumeYear{8}{2024} \EpigaArticleNr{10} \ReceivedOn{June 20, 2023}
\InFinalFormOn{February 2, 2024}
\AcceptedOn{February 16, 2024}

\title{Etale descent obstruction and anabelian geometry of curves over finite fields}
\titlemark{Etale descent obstruction and anabelian geometry of curves over finite fields}

\author{Brendan Creutz}
\address{School of Mathematics and Statistics, University of Canterbury, Private Bag 4800, Christchurch 8140, New Zealand}
\email{brendan.creutz@canterbury.ac.nz}

\author{Jos\'e Felipe Voloch}
\address{School of Mathematics and Statistics, University of Canterbury, Private Bag 4800, Christchurch 8140, New Zealand}
\email{felipe.voloch@canterbury.ac.nz}

\authormark{B.~Creutz and J.\,F.~Voloch}

\AbstractInEnglish{Let $C$ and $D$ be smooth, proper and geometrically integral curves over a finite field~$\F$. 
Any morphism $D \to C$ induces a morphism of \'etale fundamental groups $\pi_1(D) \to \pi_1(C)$. The anabelian philosophy proposed by Grothendieck suggests that, when $C$ has genus at least $2$, all open homomorphisms between the \'etale fundamental groups should arise in this way from a nonconstant morphism of curves. We relate this expectation to the arithmetic of the curve $C$ considered as a curve over the global function field $K = \F(D)$. Specifically, we show that there is a bijection between the set of conjugacy classes of well-behaved morphisms of fundamental groups and locally constant adelic points of $C$ that survive \'etale descent. We use this to provide further evidence for the anabelian conjecture and relate it to another recent conjecture by Sutherland and the second author.}


\acknowledgement{The authors were supported by the Marsden Fund administered by the Royal Society of New Zealand.}

\begin{document}


\maketitle

\begin{prelims}

\DisplayAbstractInEnglish

\bigskip

\DisplayKeyWords

\medskip

\DisplayMSCclass

\end{prelims}


\newpage

\setcounter{tocdepth}{1}

\tableofcontents


\section{Introduction}

For a smooth, proper and geometrically integral curve $X$ over a global field $k$, it is well known that the Hasse principle can fail. That is, $X$ may contain points over every completion of $k$, yet fail to have any $k$-rational point. All known examples of this phenomenon can be explained by a finite descent obstruction. This means that there is a torsor $f \colon Y \to X$ under a finite group scheme over $k$ such that no twist of $Y$ contains points over every completion. Since any $k$-rational point must lift to some twist of $f$, this yields an obstruction to the existence of $k$-rational points on $X$. A central question in the arithmetic of curves over global fields is to determine whether this is the only obstruction to the existence of $k$-rational points. 

This problem is expected to be very hard in general. For curves of genus $1$, it is equivalent to standard conjectures concerning the Tate--Shafarevich groups of elliptic curves. For curves of genus at least $2$ over number fields, it is known to follow from Grothendieck's section conjecture, but there are essentially no general results. For a discussion of the finite descent obstruction over number fields, we refer to \cite{Stoll}, which, despite being published over a decade ago, still conveys the state of the art. 

The situation is much more promising when $X$ is defined over a global function field, \textit{i.e.}, when $k = \F(D)$ is the function field of a smooth, proper and geometrically connected curve $D$ over a finite field $\F$. Building on work of Poonen--Voloch, see \cite{PV}, and R\"ossler, see \cite{Rossler}, \cite[Appendix]{CVnonisotriv}, the authors have recently completed a proof that finite descent is the only obstruction for all nonisotrivial curves of genus at least $2$;
see \cite{CVnonisotriv}. It thus remains to consider the situation for isotrivial curves. Recall that $X$ is called constant if it is isomorphic to the base change of a curve defined over $\F$, and that $X$ is called isotrivial if it becomes constant after base change to a finite extension of $k$.

We formulate a precise version of the conjecture that finite descent is the only obstruction to the existence of $k$-rational points on a constant curve over a global function field (Conjecture~\ref{conj:descent}) and prove the equivalence of this conjecture with an analogue of Grothendieck's section conjecture for curves over finite fields (see Theorem~\ref{thm:anabelian}). This enables us to use techniques from anabelian geometry which we combine with results of \cite{CV} to establish new instances of these conjectures. We prove that finite descent is the only obstruction to the existence of $k$-rational points for a constant curve $X \simeq C \times_\F k$ such that the Jacobian of $C$ is not an isogeny factor of the Jacobian of $D$ (see Theorem~\ref{thm:isog}).

\subsection{Main results and conjectures}

Let $C$ and $D$ be smooth, proper and geometrically integral curves over a finite field $\F$. We consider the arithmetic of the curve $C \times_\F K$ over the global function field $K := \F(D)$ (which we still denote by $C$ by abuse of notation). We denote by $\A_K$ the ring of ad\`eles of $K$ and consider the set $C(\A_{K})$ of adelic points of $C$, which is also the product $\prod_v C(K_v)$, where $v$ runs through the places of $K$, with its natural product topology. Let $\overline{C(K)}$ denote the topological closure of $C(K)$ inside $C(\A_K)$.

The definition of the set $C(\A_{K})^{\etale}$ of adelic points surviving descent by all torsors under finite \'etale group schemes over $K$ is recalled in Section~\ref{sec:desc}. We also consider the set $C(\A_K)^{\etale\text{-}\Br}$ of adelic points surviving the \'etale-Brauer obstruction; see \cite[Section 8.5.2]{PoonenRatPoints}. These are closed subsets of $C(\A_K)$ containing $\overline{C(K)}$. A special case of \cite[Conjecture C]{PV} implies that $\overline{C(K)} = C(\A_K)^{\etale\text{-}\Br}$. For any of the other containments in the sequence
\[
	C(K) \subset \overline{C(K)} \stackrel{?}{=} C(\A_K)^{\etale\text{-}\Br} \subset C(\A_K)^{\etale} \subset C(\A_K), 
\]
there are examples showing that, in general, they can be proper. The first will be proper when $C(K)$ is infinite, which occurs whenever there is a nonconstant morphism $\phi \in \Mor_\F(D,C) = C(K)$, as it may be composed with the Frobenius endomorphism of $C$. Examples where the third inclusion is proper are given in \cite[Proposition~4.5]{CV} and are accounted for by a descent obstruction coming from torsors under finite abelian group schemes that are not \'etale.

Despite this, it is still expected that the information obtained from $C(\A_K)^{\etale}$ should determine the set of rational points, as we now describe. For a place $v$, let $\F_v$ denote the residue field of the integer ring $\mathcal{O}_v \subset K_v$. We define $C(\A_{K,\F}) := \prod_vC(\F_v)$, which is a closed subset of $C(\A_K) = \prod_vC(K_v)$ admitting a continuous retraction $r \colon C(\A_K) \to C(\A_{K,\F})$ (see Section~\ref{notation1}). Define $C(\A_{K,\F})^{\etale} = C(\A_{K,\F}) \cap C(\A_K)^{\etale}$. Then $r(\overline{C(K}))$ is a closed subset of $C(\A_{K,\F})^{\etale}$. We conjecture the following.

\begin{Conjecture}\label{conj:descent}
We have	$r(\overline{C(K)}) = C(\A_{K,\F})^{\etale}$. In particular, $C(\A_{K,\F})^{\etale} = C(\F)$ if and only if the set $C(K) = \Mor_\F(D,C)$ contains no nonconstant morphisms.
\end{Conjecture}

Conjecture~\ref{conj:descent} is a nonabelian analogue of a conjecture in the number field case by Poonen, see \cite{Poonen}, in a setup first studied in~\cite{Scharaschkin}. It is equivalent, by \cite[Theorem 1.2]{CV}, to the conjecture that $\overline{C(K)} = C(\A_K)^{\etale\text{-}\Br}$. When $C$ has genus $1$, Conjecture~\ref{conj:descent} follows from the Tate conjecture for abelian varieties over finite fields. It is also known when the genera of $C$ and $D$ satisfy $g(D)< g(C)$ by \cite[Theorem~1.5]{CV}, and in some other cases where $C(K) = C(\F)$; see \cite[Theorem 2.14]{CVV}. The goal of this paper is to provide further evidence for this conjecture, by relating it to anabelian geometry. 

Fix geometric points $\overline{x} \in C(\Fbar)$ and $\overline{y} \in D(\Fbar)$, where $\Fbar$ denotes an algebraic closure of $\F$, and let $\pi_1(C) := \pi_1(C,\xbar)$ and $\pi_1(D) := \pi_1(D,\ybar)$ be the \'etale fundamental groups of $C$ and $D$ with these base points. Any morphism of curves $D\to C$ induces a morphism of \'etale fundamental groups $\pi_1(D) \to \pi_1(C)$ up to conjugation by an element of the geometric fundamental group $\pi_1(\Cbar) := \pi_1(C \times_\F \Fbar,\xbar) $. Grothendieck's anabelian philosophy suggests that, when $C$ has genus at least $2$, all open homomorphisms between the \'etale fundamental groups should arise in this way from a nonconstant morphism of schemes; see \cite{ST2009,ST2011}. In Section~\ref{sec:anabelian} we define a notion of well-behaved morphisms between fundamental groups of curves (see Definition~\ref{def:wb}). We expect all open homomorphisms are well behaved, but we have not been able to prove this. 

Our main result is the following theorem, which relates the set $C(\A_{K,\F})^{\etale}$ appearing in Conjecture~\ref{conj:descent} to an object of interest in anabelian geometry. 

\begin{Theorem}[\textit{cf.}~Theorem~\ref{thm:anabelian}]\label{thm:main}
	There is a bijection \textup{(}explicitly constructed in the proof\,\textup{)} between the set $\Hom_{\pi_1(\Cbar)}^{wb}(\pi_1(D),\pi_1(C))$ of well-behaved morphisms of fundamental groups up to $\pi_1(\Cbar)$-conjugation and the set $C(\A_{K,\F})^{\etale}$ of locally constant adelic points surviving \'etale descent.
\end{Theorem}

This theorem is a strengthening of an analogous result for curves over number fields, which shows that an adelic point surviving \'etale descent gives rise to a section of the fundamental exact sequence; see \cite{Harari-Stix,Stoll}. Combining Theorem~\ref{thm:main} with the results in \cite{CV}, we prove the following.

\begin{Theorem}\label{thm:isog}
	If the Jacobian $J_C$ of\, $C$ is not an isogeny factor of\, $J_D$, then Conjecture~\ref{conj:descent} holds for $C$ and $D$.
\end{Theorem}

In addition to establishing new instances of the conjecture, this
result allows us to relate it in the case $g(D) = g(C)$ to a recent
conjecture of Sutherland and the second author, see~\cite{SV}, which
we now recall.  We embed $C$ into its Jacobian $J_C $ by a choice of
divisor of degree $1$ (which always exists by the Lang--Weil estimates
since $C$ is defined over a finite field).  The Hilbert class field is
defined as follows.  Let $\Phi\colon J_C \to J_C$ denote the
$\F$-Frobenius map. Define $H(C) := (I-\Phi)^*(C) \subset J_C$, where
$I$ denotes the identity map on $J$. Then $H(C)$ is an unramified
abelian cover of $C$ with Galois group $J_C(\F)$, well defined up to a
twist that corresponds to a choice of divisor of degree $1$ embedding
$C$ into $J_C$. Define $H_0(C):=C$, $H_1(C):= H(C)$, and successively
define $H_{n+1}(C) := H_n(H(C))$ for integers $n\ge 1$.

\begin{Conjecture}[\textit{cf.} \protect{\cite[Conjecture 2.2]{SV}}]\label{main-conj}
Let $C,D$ be smooth projective curves of equal genus at least $2$
over a finite field~$\F$. If, for each $n$, there are choices of twists
such that the $L$-function of\, $H_n(C)$ is equal to
the $L$-function of\, $H_n(D)$ for all $n \ge 0$, then $C$ is isomorphic to a conjugate of\, $D$.
\end{Conjecture}

\begin{Theorem}\label{thm:1implies2}
	Suppose $g(C) = g(D) \ge 2$ and assume Conjecture~\ref{main-conj}. Then $C(\A_{K,\F})^{\etale}\ne C(\F)$ if and only if there is a nonconstant morphism $D \to C$.
\end{Theorem}

\subsection*{Acknowledgements} The authors thank Jakob Stix for suggestions leading to the proof of Proposition~\ref{prop:open} and for a correction to Remark~\ref{rem:poorbehaviour}. 

	\section{Notation and preliminaries}\label{sec:notation}
	
	\subsection{Notation}\label{notation1}
	The set of places of the global field $K = \F(D)$ is in bijection with the set $D^1$ of closed points of $D$. Given $v \in D^1$, we use $K_v$, $\mathcal{O}_v$ and $\F_v$ to denote the corresponding completion, ring of integers and residue field, respectively. Fix a separable closure $K^{\sep}$ of $K$, and let $\Fbar$ denote the algebraic closure of $\F$ inside $K^{\sep}$. For each $v \in D^1$, fix a separable closure $K_v^{\sep}$ of $K_v$ and an embedding $K^{\sep} \hookrightarrow K_v^{\sep}$. This determines an embedding $\F_v \subset \Fbar$ and an inclusion $\theta_v\colon \Gal(K_v) \to \Gal(K)$. The embedding $\F_v \subset \Fbar$ fixes a geometric point $\overline{v} \in D(\Fbar)$ in the support of the closed point $v \in D$. The inclusions $\F \subset \F_v \subset \mathcal{O}_v \subset K_v$ endow $\mathcal{O}_v, K_v$ and the adele ring $\A_K$ with the structure of an $\F$-algebra. 
We define the locally constant adele ring $\A_{K,\F} := \prod_{v \in D^1} \F_v$. This is an $\F$-subalgebra of the adele ring $\A_K$.
	 
	 The constant curve $C \times_{\Spec(\F)} \Spec(K)$ spreads out to a smooth proper model $C \times_{\Spec(\F)} D$ over $D$. For any $v \in D^1$, this gives a reduction map $r_v \colon C(K_v) \to C(\F_v)$. Since $C$ is proper, $C(\A_K) = \prod C(K_v)$ and the reduction maps give rise to a continuous projection $r \colon C(\A_K) \to C(\A_{K,\F})$ sending $(x_v)$ to $(r_v(x_v))$.

Any locally constant adelic point $(x_v) \in C(\A_{K,\F})$ determines a unique Galois equivariant map of sets $\psi \colon D(\Fbar) \to C(\Fbar)$ with the property that $\phi(\overline{v}) = x_v$. This induces a bijection $C(\A_{K,\F}) \leftrightarrow \Map_{G_\F}(D(\Fbar),C(\Fbar))$. Moreover, a locally constant adelic point on $C$ determines, and is uniquely determined by, a map $f \colon D^1 \to C^1$ together with an embedding $\F_{f(v)} \subset \F_v$ for each $v \in D^1$ (see \cite[Lemma~2.1]{CV}).

\begin{Lemma}\label{lem:adelicptmap}
	The composition $C(K) \to C(\A_{K}) \stackrel{r}\to C(\A_{K,\F})$ is injective. Composing this with the map $C(\A_{K,\F}) \to \Map(D^1,C^1)$ induces an injective map
	$C(K)/F\to \Map(D^1,C^1)$, where $C(K)/F$ denotes the set of $K$-rational points up to Frobenius twist; i.e., $P \sim Q$ if and only if there are $m,n \ge 0$ such that $F^mP = F^nQ$.
\end{Lemma}

\begin{proof}
	The first statement follow from the fact (\textit{e.g.}, \cite[Exercise 5.17]{AGI}) that a morphism defined on a geometrically reduced variety is determined by what it does to geometric points. For the second statement, see \cite[Proposition 2.3]{Stix02}.
\end{proof}

The set $C(K)/F$ is finite by the theorem of de Franchis \cite[Chapter~8, pp.~223-224]{Lang}.
Over a finite field $\F$, there is a simpler proof. The degree of a separable map $D \to C$ is bounded by Riemann--Hurwitz. Looking  at coordinates of an embedding of $C$, it now suffices to show that there are only finitely many functions on $D/\F$ of degree bounded by some $m$. The zeros and poles of such a function have degree at most $m$ over $\F$, so there are only finitely choices for the divisor of such a function. Finally, the function itself is determined up to a scalar in $\F^*$ by its divisor, but $\F^*$ is finite by hypothesis.

\subsection{Etale descent obstruction}\label{sec:desc}

	Let $f\colon C'\to C$ be a torsor under a finite \'etale group scheme $G/K$. We use $\HH^1(K,G)$ to denote the \'etale cohomology set parameterizing isomorphism classes of $G$-torsors over $K$ (and similarly with $K$ replaced by $K_v, \calO_v, \F_v$, etc.). The distinguished element of this pointed set is represented by the trivial torsor. 
	
	Following the terminology in~\cite{Stoll}, we say an adelic point $(x_v) \in C(\A_K)$ survives $f$ if the element of $\prod_v \HH^1(K_v,G)$ given by evaluating $f$ at $(x_v)$ lies in the image of the diagonal map
        $$
        \HH^1(K,G) \stackrel{\prod{\theta_v^*}}\To \prod_{v\in D^1} \HH^1(K_v,G).
        $$
        Equivalently, $(x_v)$ survives $f$ if and only if $(x_v)$ lifts to an adelic point on some twist of $f$ by a cocycle representing a class in $\HH^1(K,G)$. We use $C(\A_{K})^{\etale}$ to denote the set of adelic points surviving all $C$-torsors under \'etale group schemes over $K$. Then $C(\A_K)^{\etale}$ is a closed subset of $C(\A_K)$ containing $C(K)$. We define $C(\A_{K,\F})^{\etale} = C(\A_{K})^{\etale} \cap C(\A_{K,\F})$. By~\cite[Proposition 4.6]{CV}, an adelic point lies in $C(\A_K)^{\etale}$ if and only if its image under the reduction map  $r \colon C(\A_K) \to C(\A_{K,\F})$ lies in $C(\A_{K,\F})^{\etale}$.
	
The following lemma is a special case of a well-known statement in \'etale cohomology over a henselian ring (\textit{cf.} \cite[Remark 3.11(a) on p. 116]{Milne}).

\begin{Lemma}\label{lem:unr}
	For an \'etale group scheme $G$ over $\F$, we have $\HH^1(\mathcal{O}_v,G) = \HH^1(\F_v,G)$.
\end{Lemma}

\begin{proof}
	The canonical surjection $q \colon \mathcal{O}_v \to \F_v$ induces a map $q_* \colon \HH^1(\mathcal{O}_v,G) \to \HH^1(\F_v,G)$. This map is injective by Hensel's lemma. On the other hand, the inclusion $i\colon \F_v \to \mathcal{O}_v$ satisfies $q \circ i = \textup{id}$. It follows that $q_*$ must also be surjective. 
\end{proof}

An element of $\HH^1(K_v,G)$ is called unramified if it lies in the image of the map $\HH^1(\mathcal{O}_v,G) \to \HH^1(K_v,G)$ induced by the inclusion $\mathcal{O}_v \subset K_v$. Thus, the lemma identifies $\HH^1(\F_v,G)$ with the set of unramified elements in $\HH^1(K_v,G)$.

\section{Connection to anabelian geometry}\label{sec:anabelian}
	 
	  Fix a base point $\xbar \colon \Spec \Fbar \to \Dbar := D \times_{\Spec(\F)} \Spec(\Fbar)$. Composing with the canonical maps $\Dbar \to D$ and $D \to \Spec(\F)$, this serves as well to fix base points of $D$ and $\Spec(\F)$. The base point of $\Spec(\F)$ agrees with that determined by the algebraic closure $\F \subset \Fbar$ fixed above. This leads to the fundamental exact sequence
\begin{equation}\label{FES}
	1 \lra \pi_1(\Dbar) \lra \pi_1(D) \lra \Gal(\F) \lra 1,
\end{equation}
where $\pi_1(-)$ denotes the \'etale fundamental group with base point as chosen above. A choice of base point $\Spec\Fbar \to \Cbar$ determines a similar sequence for $C$.

The choice of separable closure of $K$ identifies $\pi_1(D)$ with the Galois group of the maximal extension $K^{\unr}$ of $K$ which is everywhere unramified. For each closed point $v \in D^1$, the embedding $\theta_v \colon \Gal(K_v) \to \Gal(K^{\sep})$ induces a section map $t_v\colon\Gal(\F_v) \simeq \Gal(K_v^{\unr}|K_v) \to \pi_1(D)$ whose image is a decomposition group $T_v \subset \pi_1(D)$ above $v$.

\begin{Definition}\label{def:wb}
	A continuous morphism $\pi_1(D) \to \pi_1(C)$ is \defi{well behaved} if every decomposition group of $\pi_1(D)$ is mapped to an open subgroup of a decomposition group of $\pi_1(C)$. Let $\Hom^{\wb}(\pi_1(D),\pi_1(C))$ denote the set of well-behaved homomorphisms of profinite groups, and for a subgroup $H < \pi_1(C)$, let $\Hom^{\wb}_{H}(\pi_1(D),\pi_1(C))$ denote the quotient of $\Hom^{\wb}(\pi_1(D),\pi_1(C))$ by the action given by composition with an inner automorphism of $\pi_1(C)$ coming from an element of $H$.
\end{Definition}

\begin{Remark}\label{rem:poorbehaviour}
	Here is an example of a poorly behaved homomorphism. Suppose the genus of $C$ is at least~$2$. By~\cite[Theorem 226]{StixBook}, there are uncountably many sections $\Gal(\F) \to \pi_1(C)$ that are not conjugate to any section coming from a point in $C(\F)$. Composing such a section with the canonical surjection $\pi_1(D) \to \Gal(\F)$ gives a continuous morphism $\pi_1(D) \to \pi_1(C)$ that is not well behaved.
\end{Remark}

\begin{Proposition}\label{prop:constructmap}
	Suppose $(x_v) \in C(\A_{K,\F})^{\etale}$. For each $v \in D^1$, let $S_v \subset \pi_1(C_{\F_v}) \subset \pi_1(C)$ be a decomposition group above the closed point $x_v \in C_{\F_v}$. Then there exists a well-behaved homomorphism $\phi\colon \pi_1(D) \to \pi_1(C)$ inducing a morphism of exact sequences
	\[
		\xymatrix{
			1 \ar[r] & \pi_1(\Dbar) \ar[d]^\phi \ar[r] & \pi_1(D) \ar[d]^\phi \ar[r] & \Gal(\F) \ar@{=}[d] \ar[r] & 1\\
			1 \ar[r] & \pi_1(\Cbar)  \ar[r] & \pi_1(C)  \ar[r] & \Gal(\F) \ar[r] & 1
		}
	\]
	such that, for each $v \in D^1$, there exists a $\gamma_v \in \pi_1(\Cbar)$ such that $\phi(T_v) = \gamma_v(S_v)\gamma^{-1}_v$.
\end{Proposition}

\begin{proof}
	For each $v \in D^1$, the choice of decomposition group $S_v \subset \pi_1(C_{\F_v})$ above $x_v$ determines a section map $s_v \colon \Gal(\F_v) \to \pi_1(C_{\F_v}) \subset \pi_1(C)$ with image $S_v$. For any finite continuous quotient $\rho_G\colon\pi_1(C) \to G$, the composition $\rho_G \circ s_v \colon \Gal(\F_v) \to G$ determines a class in $\HH^1(\F_v,G) = \Hom_G(\Gal(\F_v),G)$, the group of homomorphisms up to $G$-conjugation. Here we view $G$ as a constant group scheme over $\F$. By Lemma~\ref{lem:unr}, we may view $\HH^1(\F_v,G)$ as a subgroup of $\HH^1(K_v,G)$. In terms of descent, $\rho_G$ corresponds to a torsor in $\HH^1(C,G) = \HH^1(\pi_1(C),G) = \Hom_{G}(\pi_1(C),G)$, and $\rho_G \circ s_v$ is the evaluation of this torsor at $x_v \in C(\F_v)$. So the fact that $(x_v)$ survives \'etale descent implies that there is a global class $s \in \HH^1(K,G)$ such that for all $v \in D^1$, $\theta_v^*(s)= \rho_G \circ s_v$ in $\HH^1(\F_v,G) \subset \HH^1(K_v,G)$. Note that such an $s$ must lie in (the image under inflation of) the group $\HH^1(\pi_1(D),G) = \Hom_G(\pi_1(D),G)$ since the $s_v$ are all unramified. 
	
	For each $v \in D^1$, the condition $\theta_v^*(s) = \rho_G \circ s_v \in \HH^1(K_v,G)$ is equivalent to $s\circ t_v = \rho_G \circ s_v$ in $\HH^1(\F_v,G) = \Hom_G(\Gal(\F_v),G)$.	Let $G_v = \rho_G(\pi_1(C_{\F_v})) \subset G$ be the image of $\rho_G$ restricted to the normal subgroup $\pi_1(C_{\F_v})$. Then $G_v$ is normal in $G$ and contains the image of $\rho_G \circ s_v$, so it must also contain the image of $s\circ t_v$. Since $G$ is constant, the map $\HH^1(\F_v,G_v) \to \HH^1(\F_v,G)$ induced by the inclusion $G_v \subset G$ is injective. It follows that $\rho_G \circ s_v$ and $s\circ t_v$ are equal as elements of $\HH^1(\F_v,G_v)$.

	By the Borel--Serre theorem (see \cite[Theorem~5.12.29]{PoonenRatPoints}), the fibers of the map $\HH^1(K,G) \to \prod_{v \in D^1} \HH^1(K_v,G)$ are finite. It follows that the set
	\[
		S_G := \left\{ s' : \pi_1(D) \to G \;|\; \forall\,v\in D^1,\, s'\circ t_v = \rho_G \circ s_v \text{ in $\HH^1(\F_v,G_v)$} \right\}
	\]
	is finite, and it is nonempty by the discussion above. As in the proof of \cite[Proposition 1.2]{Harari-Stix}, it follows that the inverse limit over $G$ of these sets is nonempty.  An element of $\varprojlim S_G$ is a homomorphism $\phi \colon \pi_1(D) \to  \varprojlim G = \pi_1(C)$ with the property that for all $v \in D^1$, the maps $\phi \circ t_v$ and $s_v$ are conjugate by an element of $\pi_1(C_{\F_v}) = \varprojlim G_v$. We claim that $\phi \circ t_v$ and $s_v$ are in fact $\pi_1(\Cbar)$-conjugate. To see this, let $p \colon \pi_1(C) \to \Gal(\F)$ be the canonical surjection. Suppose $\gamma_v \in \pi_1(C_{\F_v})$ conjugates $s_v$ to $\phi\circ t_v$. We claim $\gamma_v' := \gamma_v\cdot s_v ( p(\gamma_v^{-1}))$ is an element of $\pi_1(\Cbar)$ and conjugates $s_v$ to $\phi\circ t_v$. (Note that $s_v(p(\gamma_v^{-1}))$ makes sense as $p(\gamma_v) \in \Gal(\F_v)$.) To see that $\gamma_v' \in \pi_1(\Cbar)$, we use that $p \circ s_v$ is the identity map on $\Gal(\F_v)$ to compute
	\[
			p(\gamma_v') = p\left(\gamma_v\cdot s_v\left( p (\gamma_v^{-1})\right)\right) = p(\gamma_v)\cdot (p\circ s_v)\left( p(\gamma^{-1})\right) = p(\gamma)p(\gamma^{-1}) = 1.
	\]
	To see that $\gamma_v'$ conjugates $s_v$ to $\phi\circ t_v$, we compute, for arbitrary $\sigma \in \Gal(\F_v)$, 
	\begin{align*}
		\gamma_v'\cdot s_v(\sigma)\cdot \gamma_v'^{-1} &= \left[\gamma_v\cdot s_v\left(p(\gamma_v^{-1})\right)\right] \cdot s_v(\sigma) \cdot \left[\gamma_v\cdot s_v\left(p(\gamma_v^{-1})\right)\right]^{-1}\\
		&= \gamma_v \cdot s_v\left(p(\gamma^{-1})\,\sigma p(\gamma)\right) \cdot \gamma_v^{-1}\\
		&= \gamma_v \cdot s_v(\sigma) \cdot \gamma_v^{-1},
	\end{align*}
	where the final equality uses that $\Gal(\F_v)$ is abelian.

	Finally, let us show that $\phi$ induces a morphism of exact sequences as in the statement. Write $p_D \colon \pi_1(D) \to \Gal(\F)$ for the canonical map, and use $p_C$ similarly. Since $p_C \circ s_v$ is the identity on the abelian group $\Gal(\F_v)$, for any $\sigma \in \Gal(\F_v)$, we have
	$$
		p_C(\phi (t_v(\sigma))) = p_C\left(\gamma_v\cdot s_v(\sigma)\cdot \gamma_v^{-1}\right) = p_C(s_v(\sigma)) = \sigma.
	$$
	So for any $x \in \pi_1(D)$ whose image under $p_D$ lies in $\Gal(\F_v)$, we have $p_D(x) = p_C(\phi(x))$. As this holds for all $v \in D^1$, we must have $p_D = p_C \circ \phi$. So $\phi$ induces a morphism of exact sequences as stated.
\end{proof}

\begin{Remark}
  The construction of the morphism $\phi$ in the preceding proof is similar to the proof of \cite[Proposition 1.1]{Harari-Stix}. However, the verification that it interpolates the $s_v$ up to conjugation in $\pi_1(\Cbar)$ rather than just in $\pi_1(C)$ is necessarily different from the approach in the proof of
  \cite[Proposition 1.2]{Harari-Stix}.
\end{Remark}

\begin{Construction}\label{construct2} Let $\phi \colon \pi_1(D) \to \pi_1(C)$ be a well-behaved homomorphism. From this we construct a locally constant adelic point $(x_v) \in C(\A_{K,\F})$ as follows. Let $\tilde{D}$ and $\tilde{C}$ denote the universal covers of $D$ and~$C$. The decomposition groups of $\pi_1(D)$ and $\pi_1(C)$ correspond to closed points on $\tilde{D}$ and $\tilde{C}$. As we have assumed $C$ to be hyperbolic, the intersection of any two distinct decomposition groups of $\pi_1(C)$ is open in neither (see for example \cite[Proposition 1.5]{ST2011}). So the well-behaved map $\phi$ determines a map $\tilde{\phi}\colon \tilde{D}^1 \to \tilde{C}^1$ by declaring $\tilde{\phi}(\tilde{v})$ to be the point of $\tilde{C}$ whose corresponding decomposition group contains $\phi(\mathcal{D}_{\tilde{v}})$. Given a closed point $v \in D^1$, the embedding $\theta_v\colon\Gal(K_v)\to \Gal(K)$ determines a decomposition group $T_v$ above $v$ and consequently a pro-point $\tilde{v} \in \tilde{D}$. Define $x_v \in C(\F_v) = C_{\F_v}(\F_v)$ to be the image of $\tilde{\phi}(\tilde{v})$ on $C_{\F_v}$. Ranging over the closed points of $D$, this determines a locally constant adelic point $(x_v) \in \prod_{v \in D^1}C(\F_v) = C(\A_{K,\F})$.
\end{Construction}

\begin{Remark}\label{rem:conjg}
Note that $\pi_1(C)$ acts on the set of pro-points $\tilde{w}$ above a given $w \in C^1$ and that any two pro-points above $w \in C^1$ in the same $\pi_1(\Cbar)$-orbit have the same image on $C_{\F_w}$. It follows that the adelic point $(x_v)$ from Construction~\ref{construct2} depends on $\phi$ only up to $\pi_1(\Cbar)$-conjugacy. Similarly, the image of $(x_v)$ in $\Map(D^1,C^1)$ under the map in Lemma~\ref{lem:adelicptmap} depends on $\phi$ only up to $\pi_1(C)$-conjugacy. 
\end{Remark}

\begin{Lemma}\label{lem:desc}
	Suppose $\phi \in \Hom_{\pi_1(\Cbar)}^{\wb}(\pi_1(D),\pi_1(C))$, and let $(x_v) \in C(\A_{K,\F})$ be the locally constant adelic point given by Construction~\ref{construct2}. Then $(x_v) \in C(\A_{K,\F})^{\etale}$.
\end{Lemma}

\begin{proof}
	For $v \in D^1$, let $t_v \colon \Gal(\F_v) \to \pi_1(D)$ be the section map as defined at the beginning of this section. Define $s_v = \phi \circ t_v \colon G_{\F_v} \to \pi_1(C)$. By construction, the image of $s_v$ is a decomposition group of $\pi_1(C)$ above $x_v \in C(\F_v)$. Let $\alpha \colon C' \to C$ be a torsor under a finite group scheme $G/\F$. Then $\alpha$ represents a class in $\HH^1(C,G) = \HH^1(\pi_1(C),G(\Fbar))$, where the action of $\pi_1(C)$ on $G(\Fbar)$ is induced by the projection $\pi_1(C) \to \Gal(\F)$. The evaluation of $\alpha$ at $x_v$ is the class of $\alpha \circ s_v$ in $\HH^1(\F_v,G)$. Since $\alpha \circ s_v = \alpha \circ \phi \circ t_v = t_v^*(\alpha \circ \phi)$, we see that $\alpha \circ \phi$ lies in the images of the horizontal maps in the following commutative diagram whose vertical maps come from inflation:
	\[
		\xymatrix{
			\HH^1(K,G) \ar[r]^{\theta_v^*} & \HH^1(K_v,G) \\
			\HH^1(\pi_1(D),G) \ar[r]^{t_v^*}\ar@{^{(}->}[u] & \HH^1(\F_v,G)\rlap{.} \ar@{^{(}->}[u]
			}
	\]
	As this holds for all $v \in D^1$, we see that the evaluation of $\alpha$ at the adelic point $(x_v)$ lies in the diagonal image of $\HH^1(K,G)$.
\end{proof}

\begin{Theorem}\label{thm:anabelian}
	Construction~\ref{construct2} induces bijections
	\[
		\xymatrix{
			C(\A_{K,\F})^{\etale} \ar@{->>}[d] & \Hom^{\wb}_{\pi_1(\Cbar)}(\pi_1(D),\pi_1(C)) \ar@{<->}[l] \ar@{->>}[d]\\
			\Map(D^1,C^1)^{\etale} \ar@{<->}[r] & \Hom^{\wb}_{\pi_1(C)}(\pi_1(D),\pi_1(C))\rlap{,}
			}
	\]
	where $\Map(D^1,C^1)^{\etale}$ denotes the image of\, $C(\A_{K,\F})^{\etale}$ in $\Map(D^1,C^1)$ under the map in Lemma~\ref{lem:adelicptmap}. 
\end{Theorem}

\begin{proof}

Proposition~\ref{prop:constructmap} gives a map of sets
\[
	C(\A_{K,\F})^{\etale} \lra \Hom_{\pi_1(\Cbar)}^{\wb}(\pi_1(D),\pi_1(C)),
\]
while Construction~\ref{construct2} and Lemma~\ref{lem:desc} give an injective map
\[
	\Hom^{\wb}_{\pi_1(\Cbar)}(\pi_1(D),\pi_1(C)) \longhookrightarrow C(\A_{K,\F})^{\etale}.
\]
One easily checks that these maps are inverse to one another, so they are inverse bijections.

The surjectivity of the first vertical map is given in Lemma~\ref{lem:adelicptmap}, and the surjectivity of the other is immediate from the definition. One deduces the bijection in the bottom row from that in the top row using Remark~\ref{rem:conjg}.
\end{proof}

\begin{Proposition}\label{prop:open}
	Let $\phi\colon \pi_1(D) \to \pi_1(C)$ be a well-behaved morphism corresponding to a locally constant adelic point surviving \'etale descent $(x_v) \in C(\A_{K,\F})^{\etale}$ as given by Proposition~\ref{prop:constructmap}. If $(x_v) \notin C(\F)$, then $\phi$ has open image and the map $\psi \colon D(\Fbar) \to C(\Fbar)$ induced by $(x_v)$ is surjective.
\end{Proposition}

\begin{Corollary}
	Let $\phi \colon \pi_1(D) \to \pi_1(C)$ be a well-behaved homomorphism. The image of $\phi$  either is open or is a decomposition group above a point $v \in C(\F)$.
\end{Corollary}

\begin{proof}
	Suppose the image of $\phi$ is not open. Then we find a sequence of open subgroups  $U_i \subset \pi_1(C)$ of index approaching infinity all of which contain the image of $\phi$. By Proposition~\ref{prop:constructmap}, the image of $\phi$ maps surjectively onto $\Gal(\F)$ under the canonical map $\pi_1(C) \to \Gal(\F)$. Hence, the induced maps $U_i \to \Gal(\F)$ are surjective, so that the $U_i$ correspond to geometrically connected \'etale coverings $C_i \to C$ of genus approaching infinity. For each we have a well-behaved homomorphism $\pi_1(D) \to U_i = \pi_1(C_i)$. By Theorem~\ref{thm:anabelian}, these correspond to unobstructed adelic points $(x_v^{(i)}) \in C_i(\A_{K,\F})^{\etale}$ which lift $(x_v) \in C(\A_{K,\F})$. Eventually $g(C_i) > g(D)$, in which case \cite[Theorems 1.2, 1.3 and 1.5]{CV} imply that $C_i(\A_{K,\F})^{\etale} = C_i(\F)$. But then $(x_v) \in C(\F)$. Therefore, if $(x_v)$ is nonconstant, then $\phi$ must have open image. In this case, the image of $\phi$ contains a finite-index subgroup of each decomposition group. This implies that $\psi \colon D(\Fbar) \to C(\Fbar)$ is surjective.
\end{proof}

\section{Proofs of the theorems in the introduction}

\subsection{Proof of Theorem~\ref{thm:isog}}

Suppose $(x_v) \in C(\A_{K,\F})^{\etale} \setminus C(\F)$. By Proposition~\ref{prop:open}, the Galois equivariant map $\psi \colon D(\Fbar) \to C(\Fbar)$ induced by $(x_v)$ is surjective. By~\cite[Corollary 5.3]{CV}, this induces a surjective $G_\F$-equivariant homomorphism $\phi_*\colon J_D(\Fbar) \to J_C(\Fbar)$. For any $\ell \ne p$, this yields a surjective homomorphism of the $\ell$-adic Tate modules of $T_\ell(J_D) \to T_\ell(J_C)$, so $J_C$ is an isogeny factor of $J_D$ by the Tate conjecture for abelian varieties over finite fields; see \cite{Tate}.

\subsection{Proof of Theorem~\ref{thm:1implies2}}
Let $x = (x_v) \in C(\A_K)^{\etale} \setminus C(\F)$. Since $H(C) \to C$ is an \'etale cover, $x$ lifts to a twist of $H(C)$ by an element $\xi \in \HH^1(K,J_C(\F)) = \Hom(G_K,J_C(\F))$. Let $L/K$ be the fixed field of $\ker(\xi)$. Then $L/K$ is unramified since, locally, it is given as the extension
generated by the roots of $(I-\Phi)(y) = x_v$, and $L/K$ is abelian since $\Gal(L/K)$ is a subgroup of $J_C(\F)$. Thus $L$ is a subfield of the function field $K'$ of $H(D)$ (for a suitable embedding $D \to J_D$). Viewing $x$ as an adelic point on $C$ over $K'$, we  have $x \in C(\A_{K'})^{\etale}$ by \cite[Proposition~5.15]{Stoll}. By the above, this adelic point lifts to $H(C)(\A_{K'})^{\etale}$.

From Theorem~\ref{thm:isog}, we get that  $H(C)$ and $H(D)$ have the same $L$-function. 
Now we are in the same situation as before with $H(C)$, $H(D)$ in place of $C$, $D$. Iterating this process, we obtain towers such that $H_n(D)$ and $H_n(C)$ have the same $L$-functions. Assuming Conjecture~\ref{main-conj}, this implies $C(K) \ne C(\F)$.

\begin{Remark}
The paper \cite{BoV} proves a theorem very close in spirit to Conjecture~\ref{main-conj} using $L$-functions with characters. It would be very desirable to have a proof of Conjecture~\ref{conj:descent} in the equigenus case from the main theorem of \cite{BoV} along the lines of the above proof, but we have not succeeded in producing it.
\end{Remark}



\end{document}